\documentclass{amsart}
\usepackage{amsaddr}
\usepackage{amsmath,amsthm,amssymb}
\usepackage{graphicx}
\usepackage{algorithm2e}
\usepackage{times}
\usepackage[all,cmtip]{xy}
\usepackage{tikz-cd}
\usepackage[utf8]{inputenc}
\usepackage{mathrsfs}
\usepackage{enumerate}

\DeclareSymbolFont{bbold}{U}{bbold}{m}{n}
\DeclareSymbolFontAlphabet{\mathbbold}{bbold}

\newcommand{\D}{\mathbf{d}}
\newcommand{\R}{\mathbf{r}}
\newcommand{\C}{\mathcal{C}}

\newcommand{\x}{\bowtie}
\def\st{\operatorname{star}}

\def\A{{\mathcal A}}

\def\C{{\mathcal C}}

\def\I{{\mathcal I}}
\def\Z{{\mathbb Z}}

\def\curlyA{{\mathcal A}}
\def\curlyB{{\mathcal B}}

\def\curlyF{{\mathcal F}}

\def\curlyH{{\mathcal H}}
\def\curlyR{{\mathcal R}}
\def\curlyL{{\mathcal L}}

\def\curlyT{{\mathcal T}}

\def\curlyK{{\mathcal K}}
\def\curlyP{{\mathcal P}}
\def\curlyQ{{\mathcal Q}}

\def\curlyZ{{\mathcal Z}}

\def\Q{{\mathcal Q}}
\def\P{{\mathcal P}}
\def\ra{\rightarrow}

\def\st{\operatorname{star}}

\def\varep{\varepsilon}
\def\ph{\varphi}
\def\<{\langle}
\def\>{\rangle}
\def\leq{\leqslant}
\def\geq{\geqslant}
\def\ol{\overline}

\def\fis{\operatorname{FIS}}
\def\fim{\operatorname{FIM}}

\def\Cay{\operatorname{Cay}}

\def\SchL{\operatorname{Sch^{\curlyL}}}

\def\Sq{\operatorname{Sq}}

\def\dom{\mathbf{d}}
\def\ran{\mathbf{r}}

\def\st{\operatorname{star}}

\def\relrho{\; \rho \;}

\def\relrhomin{\; \rho_{\min} \;}

\def\id{\mathbf{1}}

\newcommand{\rev}[1]{#1^{\mathrm{R}}}

\newcommand{\starsq}{\st_e^{\x}(\pi(\Sq(\curlyP),\curlyT))}

\makeatletter
\newcommand{\eqnum}{\refstepcounter{equation}\textup{\tagform@{\theequation}}}
\makeatother

\newtheorem{thm}{Theorem}[section]
\newtheorem{prop}[thm]{Proposition}
\newtheorem{lemma}[thm]{Lemma}
\newtheorem{cor}[thm]{Corollary}
\newtheorem*{thm*}{Theorem}
\newtheorem*{cor*}{Corollary}

\theoremstyle{plain} 
\newcommand{\thistheoremname}{}
\newtheorem{genericthm}[thm]{\thistheoremname}

\theoremstyle{definition}
\newtheorem{defn}{Definition}[section]
\newtheorem*{defn*}{Definition}
\newtheorem*{prop*}{Proposition}
\newtheorem{ex}[thm]{Example}
\newtheorem*{ex*}{Example}
\newtheorem*{rem*}{Remark}

\numberwithin{equation}{section}

\begin{document}

\title[Rewriting for presentations of inverse monoids]{The algebra of rewriting for presentations of inverse monoids}

\author{N.D. Gilbert and E.A. McDougall}
\address{Department of Mathematics and the Maxwell Institute for the Mathematical Sciences, Heriot-Watt University, Edinburgh, EH14 4AS}
\email{N.D.Gilbert@hw.ac.uk \textrm{(corresponding author)}}


\thispagestyle{empty}



\maketitle

\begin{abstract}
We describe a formalism, using groupoids, for the study of rewriting for presentations of inverse monoids, that is based on the 
Squier complex construction for monoid presentations.  We introduce the class of pseudoregular groupoids, an example of which now arises as the fundamental groupoid of our version of the Squier complex.  A further key ingredient is the factorisation of the presentation map from a free inverse monoid as the composition of an idempotent pure map and an idempotent separating map. 
The relation module of a presentation is then defined as the abelianised kernel of this idempotent separating map.  We then use the properties of idempotent separating maps to derive a free presentation of the relation module.  The construction of its kernel - the module of identities - uses further facts about pseudoregular groupoids.
\end{abstract}




\section*{Introduction}
\label{intro}
Inverse semigroups (and inverse monoids) comprise a class of algebraic structures that sit naturally between
the class of semigroups and the class of groups, and are the natural candidates for semigroups that are
structurally closest  to groups.  However, inverse semigroup presentations do not sit quite so naturally
between semigroup presentations and group presentations, but have particular features that set them apart.
For example, a finitely generated free inverse semigroup is not finitely presented as a semigroup
\cite{Schein}, does not have a regular language of normal forms \cite{CutSol}, and no free inverse monoid
has context-free word problem \cite{Bro}.

In this paper we consider presentations of inverse monoids as rewriting systems, and attempt to replicate the
formalism for describing rewriting in monoid presentations due to Squier \cite{Sq,SOK}, and for group
presentations due to Cremanns-Otto \cite{CrOt} and Pride \cite{Pr2}.    Given a monoid presentation $\P$ with
generating set $A$, Squier associates to $\P$ a graph $\Gamma$ that has vertex set $A^*$ (the free monoid on $A$) and,
for all $p,q \in A^*$, has an edge from $puq$ to $pvq$ whenever $u=v$ is a relation in $\P$.  A path in $\Gamma$
therefore corresponds to a chain of equivalences betwen words in $A^*$  as consequences of the relations in $\P$, and a 
homotopy relation is imposed to identify paths corresponding to such equivalences that are naturally considered
to be essentially the same.  If this homotopy relation is finitely generated, then $\P$ is said to have \textit{finite
derivation type}.  For monoid presentations of groups, a theorem of Squier \cite[Theorem 4.3]{SOK} shows that if one finite presentation of $G$ has finite derivation type then
all finite presentations of $G$ do.  The main result of \cite{CrOt} is that finite derivation type (for groups) is equivalent to the
homological finiteness property $\text{FP}_3$.

An important component of the treatment of groups (given by monoid presentations) in \cite{CrOt}  is the way in
which free reductions are handled within the formalism.  An approach based on the categorical algebra of monoidal
groupoids and crossed modules, and refashioning the results of \cite{Pr2}, was given in \cite{Gi1}.  This approach is refined and
extended in \cite{GiMcD}.  In any similar approach to presentations of inverse monoids,
we encounter the problem of handling the Wagner congruence (see \cite{Steph}, for example), which defines the 
free inverse monoid as a quotient of a free monoid, and as mentioned above, is not finitely generated.
To get around this problem , given a presentation $\P = [X:\curlyR]$ of an inverse monoid $M$, we 
define a $2$--complex as in \cite{Ki,Pr1} whose edges encode the applications of relations, and whose $2$--cells
impose an appropriate homotopy relation, but we take as vertex set an inverse monoid $\curlyT$ constructed 
canonically from $\curlyP$.  The
presentation map $\fim(X) \ra M$ from the free inverse monoid on $X$ to $M$ factors through $\curlyT$,
which has $M$ as an idempotent separating image.  We then work with the fundamental groupoid of this $2$--complex:
the use of groupoids in this general setting originates with the work of Kilibarda \cite{Ki}. As a groupoid whose set of identities
is an inverse monoid, our fundamental groupoid is an example of a \textit{pseudoregular} groupoid, whose properties are
considered in section \ref{semipseudo}.  We then aim to connect the structure of the fundamental groupoid with the 
\textit{relation module} of $\P$: for group presentations the $\text{FP}_3$ condition is equivalent to finite presentation
of the relation module.

We define the relation module of $\P$ in section \ref{rel_mod_inv_mon}.  We take a more direct approach than in
earlier work of the first author \cite{G3}, since the relation module can now be naturally defined in terms of the map
$\curlyT \ra M$, and as in \cite{G3} we show that the relation module is isomorphic to the first homology of the Sch\"utzenberger graph of $(M,X)$.  In section \ref{sq_cx_inv_mon} we establish the connection between the relation module and the fundamental
groupoid of our $2$--complex. We use an intermediate construction of a free crossed module of groupoids,
and derive a free presentation of the relation module as an $M$--module.

\section{Background notions and notation}
\label{backnot}
Our basic reference for the theory of inverse semigroups is Lawson's book \cite{LwBook}.  Aspects of the theories of groups and 
inverse semigroups are considered by side-by-side in \cite{Me}.  We shall also make use of other algebraic constructions that may be less familiar, and we give brief introductions here.

\subsection{Groupoids}
\label{ord_gpds}
A {\em groupoid} $G$ is a small category in which every morphism
is invertible.  We consider a groupoid as an algebraic structure
(as in \cite{HiBook,LwBook}) whose elements are its morphisms, with a partial associative partial binary operation
given by composition of morphisms.
The set of vertices of $G$ is denoted $V(G)$, and for each vertex $x \in V(G)$ there exists an identity morphism $1_x$.
An element $g \in G$ has domain
$g \dom$ and range $g\ran$ in $V(G)$, with $gg^{-1} = 1_{g \dom}$ and $g^{-1}g = 1_{g \ran}$.  
For $e \in V(G)$ the {\em star} of $e$ in $G$ is the set
$\st_{e}(G) = \{ g \in G : g \dom =e \}$, and the {\em local group} at $e$ is the set
$G(e) = \{ g \in G : g \dom = e = g \ran \}$.  

\begin{ex}
Let $X$ be any set and let $\rho \subseteq X \times X$ be an equivalence relation on $X$.  Then $\rho$ is a groupoid
with vertex set $X$, and with partial composition $(a,b)(c,d)=(a,d)$ if $b=c$.  If $\rho = X \times X$ we obtain the
\textit{simplicial groupoid} on $X$.
\end{ex}

\begin{ex}
Let $X$ be a topological space and $A$ a subspace of $X$.  Then the set of fixed-end-point homotopy classes of
paths in $X$ with end-points in $A$ is a groupoid, the \textit{fundamental groupoid} $\pi(X,A)$.  We shall make use of 
the fundamental groupoid $\pi(X,A)$ of a $2$--complex $X$, with $A$ its $0$--skeleton, in section \ref{sq_cx_inv_mon}.
\end{ex}

\begin{ex}
\label{trace_gpd}
An inverse semigroup $S$ may be considered as a groupoid $\vec{S}$, with $V(\vec{S})$ equal to the set of 
idempotents $E(S)$ of $S$.  The groupoid composition $\circ$ on $\vec{S}$ is the \textit{restricted product}
on $S$: the composition $s \circ t$ is defined if and only if $s^{-1}s=tt^{-1}$, and then $s \circ t = st \in S$.
This point of view is an important theme in \cite{LwBook}.
\end{ex}

\subsection{Clifford Semigroups}\label{Clifford}
Clifford semigroups constitute a class of inverse semigroups that will be of importance in the description of relation modules
in section \ref{rel_mod_inv_mon}.

Let $(E,\leq)$ be a meet semilattice, and let $\{G_{e}:\,e\in E\}$ be a family of groups indexed by the elements of $E$.  For each pair $e,f \in E$ with $e \geq f$, let $\phi^e_f:G_{e}\rightarrow G_{f}$ be a group homomorphism, and suppose that the following two axioms hold:
\begin{itemize}
\item $\phi^e_e$ is the identity homomorphism on $G_e \,,$ 
\item if $e \geq f \geq g$ then $\phi^e_f \,\phi^f_g = \phi^e_g\,.$ 
\end{itemize}
The collection 
$$(G_{e},\phi^e_f)=(\{G_{e}:\,e\in E\},\{\phi^e_f:\,e,f\in E,\,f \leq e\})$$
is a \textit{presheaf of groups over $E$} and the group operations on the $G_e$ make the disjoint union $G = \bigsqcup_{e \in E} G_e$
into an inverse semigroup, called a \textit{Clifford} semigroup
over $E$, with binary operation
$$x \ast y= (x\phi^e_{ef})(y\phi^f_{ef}) \in G_{ef}\,,$$
where $x\in G_{e}$ and  $y\in G_{f}$.

Our description of relation modules in section \ref{rel_mod_inv_mon} also depends on the factorization of an inverse semigroup homomorphism
from a free inverse monoid as a composition of an idempotent pure map and an idempotent separating map.  We recall
the definitions of these types of map here:
\begin{defn}
\label{id_pure_sep}
\leavevmode
\begin{enumerate}
\item[(a)] A congruence $\rho$ on an inverse semigroup $T$ is said to be \textit{idempotent pure} if $a \in T$ and $a \relrho e$ for some $e \in E(T)$ imply that $a \in E(T)$.
\item[(b)] A congruence $\rho$ on an inverse semigroup $T$ is said to be \textit{idempotent separating} if $e,f \in E(T)$ 
and $e \relrho f$ imply that $e=f$.
\end{enumerate}
Any inverse semigroup homomorphism $\phi : T \ra S$ 
induces a congruence $\chi_{\phi}$ on $T$ by
\[ a \; \chi_{\phi} \; b \; \iff \; a \phi = b \phi \,. \]
We say that $\phi$ is idempotent pure (respectively, idempotent separating) if $\chi_{\phi}$ has this property.
The \textit{kernel} of an inverse semigroup homomorphism $\phi : T \ra S$ is the preimage of $E(S)$:
\[ \ker \phi = \{ a \in T : a \phi \in E(S) \}. \]
\end{defn}
We recall that any inverse semigroup $T$ has a maximum group image $\widehat{T}$ and that $T$ is \textit{$E$--unitary}
if the quotient map $\sigma_T : T \ra \widehat{T}$ is idempotent pure.  Free inverse monoids are $E$--unitary.  See \cite[section 2.4]{LwBook} for more properties of 
$E$--unitary inverse semigroups.

The connection that we need between Clifford semigroups and idempotent separating maps is given by the following result (see \cite[Lemma 5.2.2]{LwBook}).

\begin{prop}
\label{ker_of_idsep}
If a homomorphism  $\phi : T \ra S$ of inverse semigroups is idempotent separating
then its kernel is a Clifford semigroup over $E(T)$.
\end{prop}

\subsection{Sch\"utzenberger graphs}
We shall use \textit{left} Sch\"utzenberger graphs in this paper.   Let $S$ be an  inverse semigroup generated by a set 
$X$.  There exists a \textit{presentation map} $\theta : \fis(X) \ra S$ from the free inverse semigroup on $X$ to $S$. 
The (left) Sch\"utzenberger graph $\SchL(S,X)$ has vertex set $S$, and for $x \in X$ and $s \in S$, an
$x$--labelled edge from
$s$ to $(x \theta)s$ whenever $(x^{-1}x) \theta \geq ss^{-1}$.  The connected component $\SchL(S,X,e)$ containing
the idempotent $e$ is the full subgraph on the vertex set $L_e$, the $\curlyL$--class of $e$ in $S$.
Some examples of Sch\"utzenberger graphs may be found in section \ref{relmod_egs}.

\subsection{Modules for inverse semigroups}
\label{mod_for_is}
Modules for inverse semigroups were first defined by Lausch \cite{Lau}.  

\begin{defn}
\label{Lau_mod}
Let $S$ be an inverse semigroup with semilattice of idempotents 
$E(S)$.  Consider a Clifford semigroup $\A = (A_e, \alpha^e_f)$ (see section \ref{Clifford}), in which each $A_e$
is an additively written abelian group with identity $0_e$.  The disjoint union $A = \bigsqcup_{e \in E(S)} A_e$ is a commutative inverse semigroup under the operation
\[ a \oplus b = a \alpha^e_{ef} + b \alpha^f_{ef} \]
for $a \in A_e$ and $b \in A_f$.  Then  $\A$ is an \textit{$S$--module}
\cite[section 2]{Lau} if there exists a map $A \times S \ra A$, written $(a,s) \mapsto a \lhd s$, such that 
\begin{enumerate}[(i)]
\item $(a \oplus b) \lhd s = (a \lhd s) \oplus (b \lhd s)$ for all $a,b \in A$ and $s \in S$,
\item $a \lhd st = (a \lhd s) \lhd t$ for all $a \in A$ and $s,t \in S$,
\item $a \lhd e = a \oplus 0_e$ for all $a \in A$ and $e \in E(S)$,
\item $0_e \lhd s = 0_{s^{-1}es}$ for all $e \in E(S)$ and $s \in S$.
\end{enumerate}
A \textit{free} $S$--module $\curlyF = (F_e, \phi^e_f)$ has as basis a family of sets $\curlyB = \{ B_e : e \in E(S)\}$, and 
$F_e$ is the free abelian group on the set
\[ \{ (b,s) : b \in B_f, s \in S, f \geq ss^{-1}, s^{-1}s  = e \} \,, \]
with $(b,s) \phi^e_{e'} = (b,se')$ and 
with $S$--action defined by $(b,s) \lhd t = (b,st)$, see \cite[section 3]{Lau}.
\end{defn}

\begin{lemma}
\label{ker_is_mod}
Let $\psi : T  \ra  S$ be a surjective idempotent separating homomorphism with kernel $K$.  Then
$\curlyK = \bigsqcup_{e \in E(T)} K_e^{ab}$ is an $S$--module, with $S$--action defined by 
$k \lhd s= t^{-1}kt$ for any $t \in T$ with $s=t \psi$.
\end{lemma}

\subsection{Crossed modules of groupoids}
\label{cm_of_gpds}
We now present the rudiments of the theory of crossed modules of groupoids.  For further information we refer to \cite{BHS},
and for the use of crossed modules in the theory of group presentations to \cite{BrHu}.

\begin{defn}
Let $G$ be a groupoid with vertex set $V = V(G)$.   Then a crossed $G$-module
$$C \xrightarrow{\partial} G \rightrightarrows V$$
consists of:
\begin{enumerate}
\item a disjoint union of groups $C = \bigsqcup_{e \in V} C_e$,  indexed by $V$,
\item a homomorphism $\partial$ of groupoids,
\item an action of $G$ on $C$, denoted $(c,g) \mapsto c^g$, such that an edge $g \in G$ with $g \D=e$ and $g \R = f$, acts on $c \in C_e$ with $c^g \in C_f$.
\end{enumerate}
The action of $G$ on $C$ satisfies
\begin{align}
\label{cm_gpd_1} & (c^g)\partial = g^{-1} (c \partial) g \quad \text{whenever} \;  c^g \; \text{is defined,} \\
\label{cm_gpd_2} & c^{a\partial} = a^{-1}ca \quad \text{where, for some} \;e \in V, \; a,c \in C_e.
\end{align}
\end{defn}

\begin{defn}
Consider a crossed $G$--module $$C \xrightarrow{\partial} G \rightrightarrows V$$ along with a set $R$ and a function $\omega:R \rightarrow G$ such that $\omega \D = \omega \R$.  Then $C$  is said to be the \textit{free crossed $G$-module on $\omega$} if for any crossed $G$--module  $$C' \xrightarrow{\partial'} G \rightrightarrows V$$ and function $\sigma : R \rightarrow C'$ such that $\omega=\sigma \partial'$ there exists a unique morphism of crossed $G$--modules  $\phi : C \rightarrow C'$ such that $\partial = \phi \partial'$.
\end{defn}

We sketch the construction of free crossed modules: see \cite[section 7.3]{BHS}.

\begin{prop}{\cite[Proposition 7.3.7]{BHS}}
\label{free_pcm_gpds}
Given a groupoid $G$, a set $R$ and a function $\omega: R \rightarrow G$ such that $\omega \D = \omega \R$, then a free crossed $G$-module on $\omega$ exists and is unique up to isomorphism.
\end{prop}

\begin{proof}
For each $e \in V(G)$ we define $R_e = \{ s \in R : (s \omega)\R = e = (s \omega)\D \}$ and
\[ (R \between G)_e = \{ (s, g) \in R \times G : r \in R_{gg^{-1}}, g^{-1}g=e \} \,. \]
We define $F_e$ to be the free group on $(R \between G)_e$, and $\mathcal{F}= \bigsqcup_{e \in V} F_e$. Then we have
a map $\delta: \mathcal{F} \rightarrow G$, defined on generators by $(s, g) \mapsto g^{-1}(s\omega)g$, and an action of $G$ on $\mathcal{F}$, defined on generators by $(s, g)^h = (s, gh)$ whenever $g^{-1}g=hh^{-1}$.
We let $P_e$ denote the
subgroup of $F_e$ generated by the elements of the form $\langle u,v \rangle = u^{-1}v^{-1}uv^{u\delta}$, for $u,v \in F_e$.
Then $P_e$ is normal in $F_e$, invariant under the $G$--action, and contained in the kernel of $\delta$.
So $\delta$ induces $\partial : \bigsqcup_{e \in V} F_e/P_e \ra G$ and this is a free crossed module on $\omega$. Uniqueness follows from the usual universal argument.
\end{proof}

We note that there exists a function $\nu : R \ra C$ induced by mapping $s \in R_e$ to $(s,e) \in (R \between G)_e$,
and that $\nu \partial = \omega$.

\subsubsection{Modules and crossed modules}
\label{mod_and_cm}

\begin{defn}
Consider a crossed module $C \xrightarrow{\partial} G \rightrightarrows V$ in which $\partial$ is trivial: that is, $\partial$ maps each $a \in C_e$ to $e \in V$.  We write $\partial = \varep$.  By CM2 each $C_e$ is then abelian, and $C$ is a 
\textit{$G$--module}.  The concept of a \textit{free}
$G$--module then follows: given a set $R$ and a function $\omega : R \ra G$ with $\omega \D = \omega \R$, a $G$--module
$\curlyA$ is \textit{free on} $\omega$, if for any $G$--module $\curlyB$ and function $\nu : R \ra \curlyB$ such that
$\nu \varep = \omega$, there exists a unique morphism $\phi : \curlyA \ra \curlyB$ of $G$--modules.
\end{defn}

More generally, we have:

\begin{prop}
\label{free_cm_and_free_mod}
\leavevmode
\begin{enumerate}
\item Let $\C \xrightarrow{\partial} G \rightrightarrows V$ be a crossed $G$--module, and let $\curlyQ$ be the 
quotient groupoid $G / \C \partial$, with $\pi : G \ra \curlyQ$ the natural map. Then $\C^{ab} = \bigsqcup_{e \in V} C_e^{ab}$ is a $\curlyQ$--module, where for $c \in C_e$ and $q = g \pi$ with $gg^{-1}=e$ we have 
\[ \ol{c} \lhd q = \ol{c^g} \,.\]
\item If $\C \xrightarrow{\partial} G \rightrightarrows V$ is a free crossed module  with basis $\omega : R \ra G$ then
$\C^{ab}$ is a free $\curlyQ$--module with basis the image of the induced map $R \ra \C \ra \ol{\C}$.
\end{enumerate}
\end{prop}

\begin{proof}
The claimed $\curlyQ$--action is well-defined, since if $q=g \pi = h \pi$ with $g,h \in G$, then $h = (a \partial)g$
for some $a \in \C$, and then by CM2,
\[ \ol{c^h} = \ol{c^{(a \partial)g}} = \ol{(a^{-1}ca)^g} = \ol{c^g} \,.\]
  
Now let $\mathcal{A}$ be an arbitrary $\curlyQ$-module, and consider the disjoint union of groups $\Lambda = \sqcup_{e \in V} \Lambda_e$, where  $\Lambda_e = C_e \partial \times A_e$.  We let $G$ act on $\Lambda$ by conjugation on each $C_e$ and via $\pi$ on $A_e$.
Let $p_1 : \Lambda \ra G$ be the projection map: we claim that 
$\Lambda \xrightarrow{p_1} G \rightrightarrows V$ is a crossed $G$--module.  For $(c \partial,a) \in \Lambda_e$
and $g \in G$ with $g \D =e$ we have:
\[
((c\partial, a )^g)p_1 = (g^{-1}(c\partial)g, a^{g\pi})p_1 =g^{-1}(c\partial)g=g^{-1}(c\partial, a)p_1g \,,\]
and for $(c_1\partial, a_1), (c_2 \partial, a_2) \in \Lambda_e$,
\begin{align*}
(c_1\partial, a_1)^{(c_2\partial, a_2)p_1} &= (c_1\partial, a_1)^{c_2\partial}\\
&= ( (c_2\partial)^{-1}(c_1\partial)(c_2\partial), a_1^{(c_2\partial) \pi} )\\
&= \left( (c_2\partial)^{-1}(c_1\partial)(c_2\partial), a_1 \right)
\intertext{since $c_2 \partial \pi = e$, and}
&= (c_2 \partial,a_2)^{-1}(c_1 \partial, a_1)(c_2 \partial, a_2) \,,
\end{align*}
since $A_e$ is abelian.  So $\Lambda \xrightarrow{p_1} G \rightrightarrows V$ is a crossed $G$--module.

Now given $\nu' :R \rightarrow \mathcal{A}$, we define 
$$\nu'' = (\nu\partial, \nu'): \mathcal{R} \rightarrow \Lambda \,.$$
We note that $\nu'' p_1 = \nu \partial$, and so by freeness of $\C$, there is an induced morphism $\lambda: \C \ra \Lambda$
of crossed $G$--modules, with $\nu \lambda = \nu''$.  Composing $\lambda$ with the second projection 
$p_2 : \Lambda \ra \A$ gives a morphism $\C \ra \A$ that factors through $\C^{ab} \ra \A$, and is easily seen to be a map
of $\curlyQ$--modules.

The maps used in the proof are illustrated below.

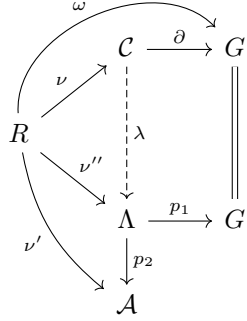
\[
\begin{tikzcd}
& \C \ar[r,"\partial"] \ar[dd,"\lambda",dashrightarrow] & G  \ar[dd, equal]\\
R \ar[ur, "\nu"] \ar[dr,"\nu''"] \ar[ddr,"\nu'"', bend right=20] \ar[urr, "\omega", bend left=70]\\
& \Lambda \ar[r,"p_1"] \ar[d,"p_2"] & G \\
& \curlyA
\end{tikzcd}
\]
\end{proof}

\section{Semiregular and pseudoregular groupoids}
\label{semipseudo}
We now introduce some additional structure on a groupoid, originating in work of Brown and Gilbert \cite{BrGi},
and further developed by Gilbert in \cite{Gi1} and by Brown in \cite{RBr}.  

\begin{defn} \label{reg_cond}
Let $G$ be a groupoid, with object set $V(G)$ and domain and range maps $\D, \R: G \rightarrow V(G)$. 
Then $G$ is \textit{semiregular} if
\begin{itemize}
\item $V(G)$ is a monoid, with identity $e \in V(G)$,
\item there are left and right actions of $V(G)$ on $G$, denoted $x \rhd \alpha$, $\alpha \lhd x$, which for all $x, y \in V(G)$ and $\alpha, \beta \in G$ satisfy:
\begin{enumerate}[(a)]
\item $(xy) \rhd \alpha = x \rhd (y \rhd \alpha)$; $\alpha \lhd (xy) = (\alpha \lhd x) \lhd y$; $(x\rhd \alpha) \lhd y = x \rhd (\alpha \lhd y)$,
\item $e \rhd \alpha = \alpha = \alpha \lhd e$,
\item $(x \rhd \alpha) \D = x(\alpha \D)$; $(\alpha \lhd x)\D = (\alpha \D)x$; $(x \rhd \alpha) \R = x(\alpha \R)$; $(\alpha \lhd x)\R = (\alpha \R)x$,
\item $x \rhd (\alpha \circ \beta) = (x \rhd \alpha) \circ (x \rhd \beta)$; $(\alpha \circ \beta) \lhd x = (\alpha \lhd x) \circ (\beta \lhd x)$, whenever $\alpha \circ \beta$ is defined,
\item $x \rhd 1_y = 1_{xy} = 1_x \lhd y$.
\end{enumerate}
\end{itemize}
\end{defn}

From \cite[section 1]{Gi1} we have the following facts.

\begin{prop}
 \label{star op} 
\leavevmode
\begin{enumerate}[(a)]
\item Let $G$ be a semiregular groupoid.  Then there are two everywhere defined binary operations on $G$ given by:
\begin{align*} \alpha \ast \beta &= (\alpha \lhd \beta \D) \circ (\alpha \R \rhd \beta) \\
\alpha \circledast \beta &= (\alpha \D \rhd \beta) \circ (\alpha \lhd \beta \R) \,.
\end{align*}
Each of the binary operations $\ast$ and $\circledast$ make $G$ into a monoid, with identity $1_e$.
\item The binary operation $\ast$ and the monoid structure on $V(G)$ make the semiregular groupoid $G$ into a strict monoidal groupoid if and only if the operations $\ast$ and $\circledast$ on $G$ coincide.
\end{enumerate}
\end{prop}

\begin{defn}
In view of part (c) of Proposition \ref{star op}, we say that a semiregular groupoid is  \textit{monoidal} if
the operations $*$ and $\circledast$ coincide.  (Brown \cite{RBr} calls such semiregular groupoids \textit{commutative
whiskered groupoids}.)
\end{defn}

\subsection{Pseudoregular groupoids}
\label{pseudoregular}
In considering presentations of inverse monoids, we shall want to consider semiregular groupoids in
which the vertex set is an inverse monoid. 

\begin{defn}
A semiregular groupoid $G$ is a \textit{pseudoregular} groupoid if $V(G)$ is an inverse monoid.
\end{defn}

The name \textit{pseudoregular} is chosen to reflect the close structural connection between inverse monoids and pseudogroups,
which are inverse semigroups of partial homeomorphisms of topological spaces (see \cite[section 1.1]{LwBook}).

In a pseudoregular groupoid $G$, the operations, $\ast$ and $\circledast$ given in proposition \ref{star op}(a) each make $G$ into a monoid, but not necessarily an inverse monoid, as we show in the next example.

\begin{ex}
Let $\partial : T \ra G$ be a crossed
module of groups.  Add a zero $0$ to $G$ to form the inverse semigroup $G^0$ and let $0 \in G^0$ act on $T$ as the trivial endomorphism $t \mapsto 1$.  The product $G^0 \times T$
is then a pseudoregular groupoid, with the following structure.
The subset $G \times T$ is a semiregular groupoid (see \cite[Proposition 1.3(ii)]{Gi1}) with vertex set $G$,  with $(g,t)\D=g$ and $(g,t)\R = g(t \partial)$, and with composition $(g,t)(h,u) = (g,tu)$ defined when $h = h(t \partial)$.  For the additional arrows in 
$\{ 0 \} \times T$ we define $(0,t)\D=0=(0,t)\R$ and composition $(0,t) \circ (0,u) =(0,tu)$ and so the local group at $0$ is a copy of
$T$.  The left and right actions of $0 \in G^0$ are given by:
\begin{align*}
0 \rhd (g,t) &= (0,1)(g,t) = (0,t), \\
(g,t) \lhd 0 &= (g,t)(0,1) = (0,1), \\
0 \rhd (0,t) &= (0,1)(0,t) = (0,t), \\
(0,t) \lhd 0 &= (0,t)(0,1) = (0,1).
\end{align*}
Then $\partial : T \ra G^0$ is a \textit{crossed monoid} (originally \textit{mono\"{i}de crois\'e}) in the sense of
Lavendhomme and Roisin \cite[Example 1.3C]{LR}, and $G^0 \times T$ is a pseudoregular
groupoid $G$, with vertex set $G^0$.  The $\ast$--operation on $G$ recovers the semidirect product $G^0 \ltimes T$:
\[ (g,t) \ast (h,u) =  (gh,t^hu)  \]
and the operations $\ast$ and $\circledast$ coincide, but $G^0 \ltimes T$ is not inverse.  This follows from the results of \cite{Nico}, but can also be seen directly, as follows.

For any $t \in T$, the element $(0,t)$ is an idempotent in $(G^0 \ltimes T, \ast)$:
\[ (0,t) \ast (0,t) = (0,t^0t) = (0,1t) = (0,t) \,. \]
But for distinct $s,t \in T$ we have
\[ (0,s) \ast (0,t) = (0,s^0t) = (0,t) \; \text{and} \; (0,t)(0,s) = (0,s) \]
and so the idempotents in $G^0 \ltimes T$ do not commute.  Since $G \ltimes T$ is a subgroup of $G^0 \ltimes T$ and the other
elements are idempotents, $G^0 \ltimes T$ is regular (and indeed \textit{orthodox}, since $E(G^0 \ltimes T)$ is a subsemigroup).
\end{ex}

In a pseudoregular groupoid $G$, it is natural to consider $\st_e(G)$
for each idempotent $e \in V(G)$. The operation $*$ then makes $\st_e(G)$ into a semigroup.  However,
as the following example shows, the identity arrow $1_e$ at $e$ is not necessarly an identity element for $(\st_e(G),\ast)$.

\begin{ex}
\label{not_a_monoid}
Let $E$ be the semilattice $\{ 1,e,f,0 \}$ with $ef=0$ and consider the simplicial groupoid $E \times E$
with vertex set $E$, and $\D$ and $\R$ given by the projection maps.
Let $U$ be the subgroupoid of $E \times E$ defined by 
\[ U = \{ (x,y) \in E \times E : x \ne 1 \ne y \} \cup \{ (1,1) \} \,.\]
Right and left actions of $E$ on $E \times E$ are defined by multiplication:
\[ x \rhd (y,z) = (xy,xz) \quad \text{and} \quad (x,y) \lhd z = (xz,yz) \,,\]
making $U$ pseudoregular.  The $\ast$--operation is given by
\[ (u,v) \ast (x,y) = ((u,v) \lhd x)(v \rhd (x,y)) = (ux,vx)(vx,vy) = (ux,vy) \,.\]
 The star at $0$ is
$\st_0 = \{ (0,e),(0,f),(0,0) \}$,
but the identity arrow $1_0 = (0,0)$ is not an identity element in $(\st_0, \ast)$.
\end{ex}

We can, however, remedy the problem illustrated in Example \ref{not_a_monoid} by passing to a subsemigroup that does 
admit $1_e$ as an identity.  For an idempotent $e \in V(G)$ we define
\[ \st_e^{\x}(G) = \{ e \rhd \alpha \lhd e : \alpha \in \st_e(G) \} \,.\]
It is clear that the operation $*$ now makes $\st_e^{\x} (G)$ into a monoid
with identity $1_e$.  The range map $\R : G \ra V(G)$ restricts to a semigroup morphism $\R_e : \st_e^{\x} (G) \ra V(G)$
whose image is a monoid $K_e$ with identity $e$.
For $\alpha \in \st_e(G)$ we set $\alpha^{\x} = e \rhd \alpha \lhd e$ and define 
\[ \pi_e^{\x}(G) = \{\alpha^{\x} \in \st_e^{\x}(G) : (\alpha^{\x})\R=e \} \,.\] 

\begin{prop} \label{piex}
In a pseudoregular groupoid $G$,
the binary operation $\ast$ and the groupoid composition $\circ$ coincide on $\pi_e^{\x}(G)$ and under each operation $\pi_e^{\x}(G)$ is a group.  Furthermore if $G$ is monoidal, then $\pi_e^{\x}(G)$ is abelian, and the family of abelian groups 
$\pi^{\x}(G) = \{ \pi_e^{\x}(G), \; e \in E(V(G)) \}$, is a $V(G)$--module.
\end{prop}

\begin{proof}
For $\alpha^{\x}, \beta^{\x} \in \pi_e^{\x}(G)$ we have
\[ \alpha^{\x} \ast \beta^{\x} = (\alpha^{\x} \lhd (\beta^{\x})\D) \circ ((\alpha^{\x})\R \rhd \beta^{\x})\\
= (\alpha^{\x}\lhd e) \circ (e \rhd \beta^{\x})\\
=\alpha^{\x} \circ \beta^{\x} \,. \]
Since $e \rhd \alpha^{-1} \lhd e =  (e \rhd \alpha \lhd e)^{-1}$ it is clear that $\pi_e^{\x}(G)$ is a subgroup of the local
group $\pi_1(G,e)$ at $e$ in the groupoid $G$.

If $G$ is monoidal, then $\ast$ and $\circledast$ coincide,and
\[
\alpha^{\x} \circ \beta^{\x} = \alpha^{\x} * \beta^{\x}\\
= \alpha^{\x} \circledast \beta^{\x}\\
=(e \rhd \beta^{\x}) \circ  (\alpha^{\x} \lhd e)\\
= \beta^{\x} \circ \alpha^{\x} \,.
\]
So $\pi_e^{\x}$ is abelian.
Now for $e \geq f$ we define $\varphi_f^e: \pi_e^{\x}(G) \rightarrow \pi_f^{\x}(G)$ by
$\alpha^{\x} \mapsto f \rhd \alpha^{\x} \lhd f \in \pi_f^{\x}$.
Then for $\alpha^{\x}, \beta^{\x} \in \pi_e^{\x}$:
\begin{align*}
(\alpha^{\x} * \beta^{\x}) \varphi^e_f &= f \rhd (\alpha^{\x} * \beta^{\x}) \lhd f\\
&=f \rhd (\alpha^{\x} \circ \beta^{\x}) \lhd f\\
&=(f \rhd \alpha^{\x} \lhd f) \circ (f \rhd\beta^{\x} \lhd f)\\
&=(f \rhd \alpha^{\x} \lhd f) * (f \rhd\beta^{\x} \lhd f)\\
&=\alpha^{\x} \varphi^e_f * \beta^{\x} \varphi^e_f
\end{align*}
and so each $\varphi_f^e$ is a homomorphism.  Furthermore, if $e \geq f \geq g$ then
$$ \alpha^{\x} \varphi^e_f \varphi^f_g = g \rhd ( f \rhd \alpha^{\x} \lhd f) \lhd g = gf \rhd \alpha^{\x} \lhd fg = g \rhd \alpha^{\x} \lhd g = \alpha^{\x} \varphi_g^e$$
and clearly $\varphi^e_e$ is the identity on $\pi_e^{\x}$.
Therefore $(\pi_e^{\x}, \varphi^e_f)$ is a presheaf of abelian groups and a
$V(G)$--action is now given by $\alpha \lhd s = s^{-1} \rhd \alpha \lhd s \in \pi_{s^{-1}es}^{\x}$. It is easy to check that the conditions in Definition \ref{Lau_mod} for a Lausch $V(G)$--module are satisfied.
\end{proof}

\section{The relation module of an inverse monoid presentation}
\label{rel_mod_inv_mon}
Let $G$ be a group generated by a set $X$, with corresponding presentation map $\theta : F(X) \ra G$.  Let
$N$ be the kernel of $\theta$: then conjugation in $F(X)$ induces a $G$--action on the abelianisation $N^{ab}$ of
$N$, and $N^{ab}$ is the \textit{relation module}.  
As shown in \cite[Corollary 5.1]{BrHu}, the relation module is isomorphic to the first homology group of the Cayley graph $\Cay(G,X)$. 

In \cite{G3} the first author introduced relation modules for inverse monoid presentations by adapting work of Crowell
\cite{Crow} on group presentations.  It was remarked  in \cite{G3} that
\begin{quote}
Defining the relation module in this way permits the introduction of the concept in other algebraic settings where the operation of
abelianisation has no obvious counterpart.
\end{quote}
However, it turns out (as we shall see below) that we can indeed define the relation module of an inverse monoid presentation
as the abelianisation of a certain Clifford semigroup, in a precise analogy of the construction for groups.
We first describe a factorization result for inverse semigroups homomorphisms. Our discussion is based on \cite[page 265]{MaMe}, to which we refer for further details.  The result originates in \cite[Theorem 4.2]{RS}.
\begin{prop} 
\label{factor}
Let $\rho$ be  a congruence on the inverse semigroup $S$.  Then there exists a smallest congruence $\rho_{\min}$
on $S$ whose trace is the same as the trace of $\rho$, defined by
\begin{equation}
\label{defn_min_cong}
 a \, \rho_{\min} \, b \iff \; \text{there exists} \; e \in E(S) \; \text{with} \; ae=be \; \text{and} \; a^{-1}a \relrho e \relrho b^{-1}b \,.
\end{equation}
Furthermore,
\leavevmode
\begin{enumerate}
\item[(a)] For $a, b \in S$ we have 
\begin{equation}
\label{alt_defn_min}
a \relrhomin b \; \iff \; \text{there exists} \; c \in S \; \text{with} \; a \geq c \leq b \; \text{and} \;a \relrho c \relrho b.
\end{equation}
\item[(b)] The canonical map $\psi: S / \rho_{\min} \ra S / \rho$ is idempotent separating,
\item[(c)] If $S$ is $E$--unitary then the canonical map $\tau : S  \ra S /  \rho_{\min}$ is idempotent pure.
\end{enumerate}
\end{prop}

\begin{proof}
(a) We first show that the conditions \eqref{defn_min_cong} and \eqref{alt_defn_min} are equivalent.  
First assume that  \eqref{defn_min_cong} holds and set $c=ae=be$.
Then $a \geq c \leq b$ and, since $a^{-1}a \relrho e \relrho b^{-1}b$ we have
\[ a = aa^{-1}a \relrho ae = be \relrho bb^{-1}b = b \,.\]
Now if \eqref{alt_defn_min} holds, take $e = c^{-1}c$.  Since $a \geq c \leq b$ we have $ae = c = be$, and since $a \relrho c$ we have
$a^{-1}a \relrho c^{-1}c = e$.  Similarly $b^{-1}b \relrho c$.

(b) Suppose that $a,b \in S$ with $a \relrhomin a^2$ and $b \relrhomin b^2$.  By Lallement's Lemma \cite[Lemma 2.4.3]{HoBook}, there
exist $e,f \in S$ with $a \relrhomin e$ and $b \relrhomin f$.  If now $e \relrho f$ then $e \relrhomin f$,
and so $a \relrhomin b$.  Hence $\psi$ is idempotent separating.

(c) Suppose that, for $s \in S$ and $x \in E(S)$,  we have $s \relrhomin x$. Then there exists $e \in E(S)$
with $se = xe$ and $xe \in E(S)$, and if $S$ is $E$--unitary, we have $s \in E(S)$ and so $\tau$ is idempotent pure.
\end{proof}

We now consider an inverse monoid presentation $\P = [X:\curlyR]$ of an inverse monoid $M$.  We set $A = X \sqcup X^{-1}$,
and so $M$ is then a quotient of the free monoid $A^*$, with canonical map $\varphi : A^* \ra M$, and also a quotient of the free inverse monoid $\fim(X)$, with associated presentation map
$\theta : \fim(X) \ra M$. The Wagner congruence on $A^*$ induces the natural map $\rho: A^* \ra \fim(X)$,
and $\varphi = \rho \theta$, and we may factorize $\theta$ as in Proposition \ref{factor}.  We set $\curlyT(M,X) = \fim(X) / \theta_{\min}$
and so have the commutative diagram
\begin{equation} \label{factored}
\xymatrix@=4em{
A^* \ar[r]^{\rho} \ar@/^3.0pc/[rrr]_{\varphi} & \fim(X) \ar[r]^{\tau} \ar@/_2.0pc/[rr]_{\theta} & \curlyT(M,X) \ar[r]^{\psi} & M } \end{equation}
Since $\fim(X)$ is $E$--unitary, the map $\tau$ is idempotent pure and we obtain from \cite{MaMe}  the following  structural information on $\curlyT(M,X)$.

\begin{lemma}{\cite[Lemma 1.6]{MaMe}} 
\label{arboreal} 
Let $\mathcal{P} = [X:\curlyR]$ be a presentation of an inverse monoid $T$. Then the following are equivalent:
\begin{enumerate}[(a)]
\item the presentation map $\theta : \fim(X) \ra T$ is idempotent pure.
\item $\mathcal{P}$ is equivalent to a presentation of the form $\mathcal{P}_1 = [X; \curlyR_1]$  where $\curlyR_1 = \{ e_i = f_i : i \in I \}$ for some set $I$ and idempotents $e_i$, $f_i$ of $\fim(X)$.
\item Each Sch\"utzenberger graph, $\SchL(T,X,e)$ is a tree.
\end{enumerate}
\end{lemma}

\begin{defn}
An inverse monoid $T$ is \textit{arboreal} \cite{G3} if it satisfies the conditions of Lemma \ref{arboreal}.
\end{defn}

\begin{cor}
\label{arboreal_is_Eunitary}
An arboreal inverse monoid $T$ is $E$--unitary.
\end{cor}

\begin{proof}
It follows from part (b) of Lemma \ref{arboreal} that $T$ has maximum group image $F(X)$ and that the quotient map
$\sigma : \fim(X) \ra F(X)$ factorizes as $\tau \sigma_T$.  Since $\sigma$ is idempotent pure and $\tau$ is surjective,
the map $\sigma_T$ is idempotent pure.
\end{proof}

The factorization of $\theta$ shown in \eqref{factored} gives us an
idempotent separating homomorphism $\psi :  \curlyT(M,X) \ra M$.
By Proposition \ref{ker_of_idsep},
\[ K = \ker \psi = \{ w \in \curlyT(M,X) : w \psi \in E(M) \}. \]
is a Clifford semigroup, and so  is a union of groups $K_e$, indexed by the idempotents of $M$. Hence $K$ has the natural abelianisation 
\[ \curlyK = \bigcup_{e \in E(M)} K_e^{ab} \] 
that is an $M$--module by Lemma \ref{ker_is_mod}.  

\begin{defn}
The \textit{relation module} of the presentation $\P$ is the $M$--module $\curlyK$.
\end{defn}

We now draw the connection between relation modules and  Sch\"utzenberger graphs.
For the left Sch\"utzenberger graph $\SchL(M,X)$, the cellular chain group
$C_0(\SchL(M,X,e))$ is the free abelian group on the $\curlyL$--class $L_e$ in $M$, and    
$C_1(\SchL(M,X,e))$  is the free abelian group on the set
$$\{ (x,s) : x \in X, s \in M, (x^{-1}x) \theta \geq ss^{-1}, s^{-1}s=e \} \,.$$ 
The boundary map $\partial: C_1(\SchL(M,X,e)) \ra C_0(\SchL(M,X,e))$ maps $(x,s) \mapsto (x \theta)s-s$.
Now
\[ C_0(\SchL(M,X)) = \bigoplus_{e \in E(M)} C_0(\SchL(M,X,e)) \]
and 
\[ C_1(\SchL(M,X)) = \bigoplus_{e \in E(M)} C_1(\SchL(M,X,e)) \,,\]
and by defining 
$$s \lhd t = st \; \text{and} \; (a,s) \lhd t = (a,st)$$
we get an $M$--module structure on each of $C_0(\SchL(M,X))$ and $C_1(\SchL(M,X))$.
The boundary map
$ \partial: C_1(\SchL(M,X)) \ra C_0(\SchL(M,X))$
is then a map of $M$--modules, and its kernel $H_1(\SchL(M,X))$ is an $M$--module $\curlyH$, with the group
$H_1(\SchL(M,X))_e$ being the first homology group $H_1(\SchL(M,X,e))$ of the connected component containing $e \in E(M)$.

\bigskip
For $e \in E(M)$ we define
$U_e = \{ w \in A^* : w \varphi \geq e \}$.
Then $U_e$ is a submonoid of $A^*$ and is the reverse of the language accepted by the Sch\"utzenberger graph $\SchL(M,X,e)$ when regarded as an automaton with input alphabet
$A$ and unique start/accept state $e$, see \cite{Steph}.  (The reversal arises because we assume that as an automaton, 
$\SchL(M,X,e)$ reads input words from the left, but the action on states is by left multiplication in $M$.) Let $\pi_e$ denote the fundamental group $\pi_1(\SchL(M,X,e),e)$.  A closed path $\alpha$ in $\SchL(M,X,e)$ is labelled by a unique word $w \in A^*$
whose reverse $\rev{w}$ is in $U_e$: we write $[w]$ for the homotopy class of $\alpha$ in $\pi_e$.

\begin{lemma}
\label{wordplay}
There is a group isomorphism $\kappa_e : \pi_e \ra K_e$ mapping the homotopy class $[w] \in \pi_e$  to $[(\rev{w} \rho \tau) \tilde{e}]^{-1}$, where $\rev{w}$ is the reverse of $w$ and $\tilde{e}$ is the unique preimage in $E(\curlyT(M,X))$ of $e \in E(M)$.
\end{lemma} 

\begin{proof}
If $[w] \in \pi_e$ then $\rev{w} \in U_e$ and
\[ \psi : (\rev{w} \rho \tau) \tilde{e} \mapsto (\rev{w} \rho \tau \psi)(\tilde{e} \psi) = (\rev{w} \varphi)e = e \,. \]
Hence $[(\rev{w} \rho \tau) \tilde{e}]^{-1} \in K_e$.
To verify that $\kappa_e$ is well-defined on $\pi_e$ , suppose that $[u]=[v]$.  Then $v$ can be obtained from $u$
by the insertion and deletion of subwords $aa^{-1}$ with $a \in A$.  Considering one such step, if for some $p,q \in A^*$
we have $u=pq$ and $v=paa^{-1}q$ then
\[ [(\rev{u} \rho \tau)  \tilde{e}]^{-1} \geq [(\rev{v} \rho \tau)  \tilde{e}]^{-1} \]
in the subgroup $K_e$ of $\curlyT(M,X)$: since the relation $\geq$
is trivial on $K_e$, we deduce that $[(\rev{u} \rho \tau) \tilde{e}]^{-1} = [(\rev{v} \rho \tau) \tilde{e}]^{-1}$, and so $\kappa_e$ is well-defined. 

Now for $u,v \in U_e$ we have
\begin{align*}  [u]\cdot[v] = [uv] \mapsto [(\rev{(uv)} \rho \tau) \tilde{e}]^{-1}  &= \tilde{e}(\rev{u} \rho \tau)^{-1} (\rev{v} \rho \tau)^{-1} 
\\ & = \tilde{e}(\rev{u} \rho \tau)^{-1} \tilde{e} (\rev{v} \rho \tau)^{-1} = ([u] \kappa_e)([v] \kappa_e)\,,\end{align*}
since $\tilde{e}$ is the identity of the group $K_e$ and $\tilde{e}(\rev{u} \rho \tau)^{-1}  \in K_e$.  
Hence $\kappa_e$ is a homomorphism.

If $k \in K_e$ we set $w_k$ to be any word in $A^*$ with $\rev{w_k} \rho \tau = k$.  Then 
$e = \rev{w_k} \rho \tau \psi = \rev{w_k} \varphi$ and so $\rev{w_k} \in U_e$ and $[w_k] \in \pi_e$.
Therefore $\kappa_e$ is surjective.

Now suppose that $[w] \in \ker \kappa_e$.  Then $(\rev{w} \rho \tau) \tilde{e} = \tilde{e} \in \curlyT(M,X)$, and since
$\curlyT(M,X)$ is $E$--unitary, we deduce that $\rev{w} \rho \tau \in E(\curlyT)$.  Since $\tau$ is idempotent pure,
$\rev{w} \rho \in E(\fim(X))$ and $\rev{w}$ is freely reducible to the empty word: hence the circuit at $e$ in
$\SchL(M,X,e)$ labelled by $w$ is homotopic to the constant path at $e$, and $[w]$ is trivial.  Therefore
$\kappa_e$ is injective.
\end{proof}

\begin{thm}
The $M$--modules
\[ \curlyK = \bigsqcup_{e \in E(M)} K_e^{ab} \quad \text{and} \quad \curlyH =  \bigsqcup_{e \in E(M)} H_1(\SchL(M,X,e)) \]
are isomorphic.
\end{thm}

\begin{proof}
We identify $H_1(\SchL(M,X,e))$ with $\pi_e^{ab}$ to exploit Lemma \ref{wordplay}: for each $e \in E(M)$ there is a group isomorphism $\bar{\kappa}_e : \pi_e^{ab} \ra K_e^{ab}$.
The action of $t \in M$ on $\curlyH$ is induced by the family of maps  $\pi_e^{ab} \ra \pi_{t^{-1}et}^{ab}$, in which the
image of a closed path $[w]$ in $\pi_e^{ab}$ is mapped to the image of $[u^{-1}wu]$ in $\pi_{t^{-1}et}^{ab}$, where
$t^{-1} = \rev{u}\ph$ (and so $u$ labels a path from $et$ to $t^{-1}et$ in $\SchL(M,X,t^{-1}et)$).
We note that the isomorphism $\kappa_{t^{-1}et}$ maps
\[ [u^{-1}wu] \mapsto \rev{(u^{-1}wu)} \rho \tau \cdot \widetilde{t^{-1}et} \,,\]
where $\widetilde{t^{-1}et}$ is the unique element of $E(\curlyT(M,X))$ with $(\widetilde{t^{-1}et}) \psi = t^{-1}et$.

By Lemma \ref{ker_is_mod}, the $M$--action on $\curlyK$ is induced by conjugation in $\curlyT(M,X)$: for $k \in K_e$
with image $\bar{k} \in K_e^{ab}$,
\[ \bar{k} \lhd t = \overline{\tilde{t}^{-1} k \tilde{t}} \in K_{t^{-1}et}^{ab} \]
for any $\tilde{t}$ with $\tilde{t} \psi = t$.  We set $\tilde{t} = \rev{(u^{-1})} \rho \tau$.  Then for  $[w] \in \pi_e$,
\begin{align*}
\tilde{t}^{-1}([w] \kappa_e)\tilde{t} &= \tilde{t}^{-1} (\rev{w} \rho \tau) \tilde{e} \tilde{t} \\
&= \tilde{t}^{-1} (\rev{w} \rho \tau) \tilde{t} \tilde{t}^{-1} \tilde{e} \tilde{t} \\
&= (\rev{u} \rho \tau) (\rev{w} \rho \tau) (\rev{(u^{-1})} \rho \tau) \tilde{t}^{-1} \tilde{e} \tilde{t} \\
& = (\rev{(u^{-1}wu)} \rho \tau) \cdot \widetilde{t^{-1}et}
\end{align*}
using the fact that $\psi$ is idempotent separating.  Therefore the  diagram
$$\xymatrix@=4em{
\pi_e^{ab}  \ar[r]^{\bar{\kappa}_e} \ar[d] & K_e^{ab} \ar[d]^{\lhd t} \\
\pi_{t^{-1}et}^{ab}  \ar[r]_{\bar{\kappa}_{t^{-1}et}} & K_{t^{-1}et}^{ab} }$$
commutes, and the family of maps $\{\bar{\kappa}_e : E \in E(M) \}$ is an $M$--module isomorphism.

\end{proof}

\subsection{Examples of relation modules}
\label{relmod_egs}
\begin{ex}
Let $M$ be the semilattice $\{ 1,e,f,ef \}$, generated as an inverse monoid by $\{ e,f \}$.  The Sch\"utzenberger graph is
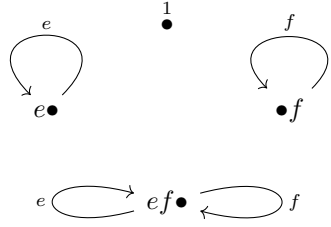
\[
\begin{tikzcd}
& \stackrel{1}{\bullet} & \\
e \bullet \ar[loop,"e"',distance=1.5cm] && \bullet f \ar[loop,"f"',distance=1.3cm]\\
& ef \bullet \ar[loop left,"e",distance=1.5cm]  \ar[loop right,"f",distance=1.5cm] &
\end{tikzcd}
\]
and the relation module is therefore 
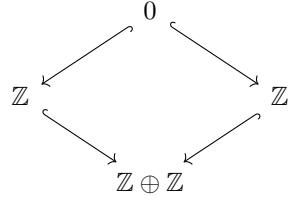
\[
\begin{tikzcd}
& 0  \ar[dr,hook] \ar[dl,hook] & \\
\Z \ar[dr,hook] && \Z \ar[dl,hook] \\
& \Z \oplus \Z &
\end{tikzcd}
\]
where all the structure maps are inclusions.  
\end{ex}

\begin{ex}
The \textit{bicyclic monoid} $B$ is the inverse monoid presented by $[x : xx^{-1}=1 ]$.  The Sch\"utzenberger graph
$\SchL(B,x,x^{-q}x^q)$ is the semi-infinite path
\[
\begin{tikzcd}
x^q & x^{-1}x^q \ar[l,"x"'] & x^{-2}x^q \ar[l,"x"'] & \dotsc \ar[l,"x"']  & x^{-k}x^q \ar[l,"x"'] & \dotsc \ar[l,"x"'] 
\end{tikzcd}
\]
The relation module $\curlyK$ is therefore trivial.  This is no surprise: $B$ is an arboreal inverse monoid, and this Example
illustrates Lemma \ref{arboreal}.   
\end{ex}

\begin{ex}
Given an inverse monoid $M$ with presentation $[Y:R]$, we add a zero to $M$ to obtain $M^0$.  For $M^0$ we take the generating
set $X = Y \cup \{ z \}$ (with $z \not\in Y$), and we have a presentation $\Q$ of $M^0$ given by 
\[ \Q = [Y,z : R, z^2=z, yz=z=zy \; (y \in Y) ] \,.\]
In the Sch\"utzenberger graph there is a loop at $0$ for each
element of $X$.  If $[Y:R]$ has relation module $\curlyK$ then the relation module of $\Q$ can be thought of schematically as
\[ \begin{tikzcd}
\curlyK \ar[dd,"\zeta"] \\ \\
\Z^{|X|}
\end{tikzcd} \]
where the map $\zeta$ carries a circuit in $\SchL(M,Y)$ labelled by a word $w \in (Y \sqcup Y^{-1})^*$ to the element of
$\Z^{|Y|} \subset \Z^{|X|}$ determined by $w$.
\end{ex}

\begin{ex}
The symmetric inverse monoid $\I_2$ on the set $\{ 1,2 \}$ is generated by the transposition $\tau$ and the 
identity map $\varep$ on $\{ 1 \}$: then $\varep \tau \varep$ is the empty map $\mathbf 0$, and
$\I_2 = \{ \id, \tau, \varep, \tau \varep, \varep \tau, \tau \varep \tau, {\mathbf 0} \} \,.$
The Sch\"utzenberger graph $\SchL(\I_2, \{ \tau, \varep\})$ is
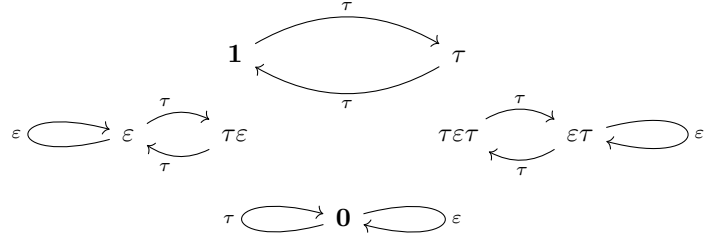
\[
\begin{tikzcd}
& \id \ar[rr, bend left, "\tau"] && \tau \ar[ll, bend left, "\tau"] & \\
\varep \ar[loop left,"\varep",distance=1.5cm] \ar[r,bend left,"\tau"] & \tau \varep \ar[l, bend left, "\tau"] &
& \tau \varep \tau \ar[r,bend left, "\tau"] & \varep \tau \ar[loop right,"\varep",distance=1.5cm] \ar[l,bend left,"\tau"] \\
&& \mathbf{0} \ar[loop left,"\tau",distance=1.5cm]  \ar[loop right,"\varep",distance=1.5cm] &&
\end{tikzcd}
\]
The relation module is therefore
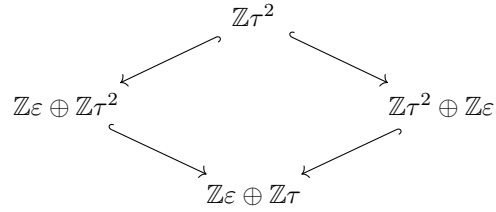
\[
\begin{tikzcd}
& \Z \tau^2 \ar[dl,hook] \ar[dr,hook] & \\
\Z \varep \oplus \Z \tau^2 \ar[dr,hook] && \Z \tau^2 \oplus \Z \varep \ar[dl,hook] \\
& \Z \varep \oplus \Z \tau &
\end{tikzcd}
\]
\end{ex}

\section{The Squier complex of an inverse monoid presentation}
\label{sq_cx_inv_mon}
In this section we show that we can obtain a presentation of the relation module $\curlyK$, derived from an inverse monoid presentation $\curlyP = [X:\curlyR]$ with presentation map 
$\theta: \fim(X) \ra M$, from a free crossed module that is in turn derived from a 
Squier complex $\Sq(\curlyP)$ associated to $\curlyP$.

\begin{defn}
Let $\curlyP = [X:\curlyR]$ be an inverse monoid presentation of $M$, with presentation map
$\theta : \fim(X) \ra M$ factorised as in \eqref{factored}, as
$$FIM(X) \xrightarrow{\tau} \curlyT(M,X) \xrightarrow{\psi} M $$
with $\tau$ idempotent pure and $\psi$ idempotent separating. We write $\curlyT$ for $\curlyT(M,X)$: then the
\textit{Squier complex} 
$\Sq(\curlyP)$ of $\curlyP$ is the $2$--complex constructed as follows.
\begin{itemize}
\item The vertex set is $\curlyT$.
\item The edge set consists of all $4$-tuples $(p,l,r,q)$ with $p,q \in  \curlyT$ and $(l,r) \in R$. Such an edge will start at $p(l \rho \tau)q$ and end at $p(r \rho \tau)q$, so each edge corresponds to the application of a relation from $\curlyP$ in $\curlyT$. An edge path in $\Sq(\curlyP)$ therefore corresponds to a succession of such applications.
\item The 2-cells correspond to applications of non-overlapping relations, and so a 2-cell is attached along every edge path of the form:
$$
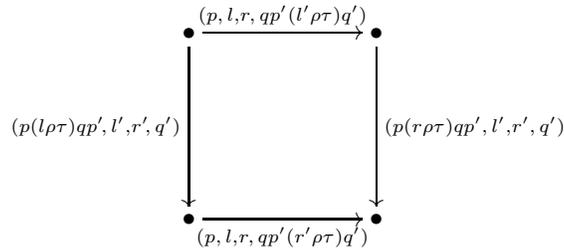
\xymatrix@=6em{
\bullet \ar[d]_{(p(l \rho \tau)qp'\!, \, l',r'\!, \, q')} \ar[r]^{(p, \, l,r , \, qp'(l' \rho \tau)q')} & \bullet \ar[d]^{(p(r \rho \tau)qp', \, l',r' , \, q')} \\
\bullet \ar[r]_{(p, \, l,r, \, qp'(r' \rho \tau)q')} & \bullet}$$
This attachment of $2$--cells makes the two edge paths between $p(l \rho \tau)qp'(l' \rho \tau)q'$ and $p(r \rho \tau)qp'(r' \rho \tau)q'$ homotopic in $\Sq(\curlyP)$. 
\end{itemize}
\end{defn}

\begin{prop}
\label{monoidal_Pi_invmon}
The fundamental groupoid $\pi(Sq(\curlyP),\curlyT)$ is pseudoregular and monoidal.
\end{prop}

\begin{proof}
The actions of $\curlyT$ on single edges in $\Sq(\P)$ given by
\[ t \rhd (p,l,r,q) = (tp,l,r,q) \; \text{and} \; (p,l,r,q) \lhd t = (p,l,r,qt) \]
induce a pseudoregular structure on the fundamental groupoid.
The $2$--cells of $\Sq(\P)$ ensure that, if $\alpha$ and $\beta$ are the homotopy classes of edge-paths of length
$1$ in $\Sq(\P)$, then $\alpha \ast \beta = \alpha \circledast \beta$, and a straightforward induction extends this to
arbitrary edge-paths.
\end{proof}

It will be convenient in what follows to describe operations in the fundamental groupoid $\pi(\Sq(\curlyP),\curlyT)$
as being performed on edge-paths in $\Sq(\curlyP)$ rather than on fixed-end-point homotopy classes.  

Let $e \in E(\curlyT)$.
Then $\st_e^{\x}(\pi(\Sq(\curlyP),\curlyT))$ has vertex set
$K_e = \{ a \in \curlyT : a\psi = e \}$,
which we recognise as one of the groups that make up the kernel of $\psi$. 

\begin{lemma} \label{st_gp}
Let $e \in E(\curlyT)$. Then $(\starsq,\ast)$ is a group.
\end{lemma}

\begin{proof}
The set $\starsq$ is a monoid under the
operation $\ast$, and for $\alpha \in \starsq$ 
we define
\[ \alpha^* = (\alpha \R)^{-1} \rhd \alpha^{\circ} \lhd (\alpha \D)^{-1} \,,\]
where a superscript $~^{-1}$ denotes the inverse in the inverse monoid
$\curlyT$ and a superscript $~^{\circ}$ denotes the inverse in the 
groupoid $\Sq(\curlyP)$. Now
\begin{align*}
\alpha * \alpha^* &= \alpha * \big( (\alpha \R)^{-1} \rhd \alpha^{\circ} \lhd (\alpha \D)^{-1} \big) \\
&= \big( \alpha \lhd (\alpha \R)^{-1} (\alpha^{\circ}\D)(\alpha \D)^{-1} \big ) \circ \big( (\alpha \R)(\alpha \R)^{-1} \rhd \alpha^{\circ} \lhd (\alpha \D)^{-1} \big)\\
&=\big( \alpha \lhd (\alpha \R)^{-1} (\alpha \R) (\alpha \D)^{-1} \big ) \circ \big( (\alpha \R)(\alpha \R)^{-1} \rhd \alpha^{\circ} \lhd (\alpha \D)^{-1} \big).
\end{align*}

Since $\alpha \in \starsq$ we have $\alpha \D=e$, and since $\alpha \R \in K_e$ and $K_e$ is a  subgroup of $\curlyT$ with identity $e$, then
$(\alpha \R)^{-1} (\alpha \R) = e = (\alpha \R)(\alpha \R)^{-1}$. So
\[ \alpha * \alpha^* = (\alpha \lhd e) \circ (e \rhd \alpha^{\circ} \lhd e) = \alpha \circ \alpha^{\circ} =1_e \,. \]
Similarly $\alpha^* * \alpha = 1_e$, and  $\starsq$ is a group.
\end{proof}

Since we shall be working
exclusively in the groupoid $\pi(\Sq(\curlyP),\curlyT)$ hereon, we shall abbreviate $\st_e^{\x}(\pi(\Sq(\curlyP),\curlyT))$ to
$\st_e^{\x}$.

\begin{lemma}
\label{inv_in_star}
Suppose that $(ep, l,r, qe) \in \st_e^{\x}$ and set $h = l \rho \tau$ and $k = r \rho \tau$.  Then $ephqe = e$ and
\[
(ep)(hqe)(ep) = eep = ep \quad \text{and} \quad
(hqe)(ep)(hqe) = hqee = hqe \,.
\]
Therefore $ep = (hqe)^{-1}$ in $\curlyT$ and  $(ep, l,r, qe) = (eq^{-1}h^{-1}, l,r, qe)$:
moreover $e=eq^{-1}h^{-1}hqe$ and so $e \leq q^{-1}h^{-1}hq$.
\end{lemma}

\begin{lemma} \label{lambda2}
A path $\alpha \in \st_e^{\x}$ can be written as a product of single edges in the group $(\st_e^{\x},\ast)$. Hence $\st_e^{\x}$ is generated by the subset $\Sigma_e^{\x}$ of homotopy classes of single edges in $\st_e^{\x}$, and these classes are
represented by edges of the form
$\lambda^e_{l,r,q} = (eq^{-1}(l^{-1}) \rho \tau,l,r,qe)$
with $e \leq q((l^{-1}l)\rho \tau) q$.
\end{lemma}

\begin{proof}
The vertex set of  the connected component of $\Sq(\curlyP)$ that contains $e$ is the group $K_e$ (with identity $e$), and for a path $\alpha$ in this component we define
\[ \alpha \lambda = (\alpha \D)^{-1} \rhd \alpha \lhd e \,. \]
If $\alpha =(p,l,r,q)$ is a single edge and $h  = l \rho \tau$, then
\[
\alpha \lambda = (q^{-1}h^{-1}p^{-1}p,l,r,qe) = (eq^{-1}h^{-1}p^{-1}p,l,r,qe) \]
since $phq \in K_e$.  Then by Lemma \ref{inv_in_star}, we have $eq^{-1}h^{-1}p^{-1}p = (hqe)^{-1}$ and so
\[ \alpha \lambda = (eq^{-1}h^{-1},l,r,qe) = \lambda^e_{l,r,q} \,, \]
and $(\alpha \lambda)\D = eq^{-1}h^{-1}hqe=e$.

Now if $\alpha = \alpha_1 \circ \alpha_2$ then
\begin{align*}
\alpha \lambda &= (\alpha_1 \D)^{-1} \rhd ( \alpha_1 \circ \alpha_2) \lhd e \\
&= ((\alpha_1 \D)^{-1} \rhd \alpha_1 \lhd e) \circ ((\alpha_1 \D)^{-1} \rhd \alpha_2 \lhd e) \\
&= \alpha_1 \lambda \circ ((\alpha_1 \D)^{-1} \rhd \alpha_2 \lhd e) \,,
\end{align*}
and 
\[ \alpha_1 \lambda \ast \alpha_2 \lambda = (\alpha_1 \lambda \lhd e) \circ ((\alpha_1 \D)^{-1}(\alpha_1 \R)e(\alpha_2 \D)^{-1} 
\rhd \alpha_2 \lhd e) \,.\]
But $\alpha_1 \R = \alpha_2 \D \in K_e$ and so $(\alpha_1 \D)^{-1}(\alpha_1 \R)e(\alpha_2 \D)^{-1} = (\alpha_1 \D)^{-1}$
and therefore $\alpha  \lambda = \alpha_1 \lambda \ast \alpha_2 \lambda$.  The Lemma then follows easily by induction on the length of a
path.
\end{proof}

Now suppose that a path $\alpha$ with $\alpha \D=e$ is a composition $\alpha = \alpha_1 \circ \alpha_2$ and that $\beta$ is the path
\[ \beta = \alpha_1 \circ \gamma \circ \gamma^{\circ} \circ \alpha_2 \]
for some path $\gamma$.  Then  \[ \beta \lambda = \alpha_1 \lambda \ast \gamma \lambda \ast \gamma^{\circ} \lambda \ast \alpha_2 
\lambda \,. \]
Now if $x = \gamma \D$ and $y= \gamma \R$ then
\begin{align*}
\gamma \lambda \ast \gamma^{\circ} \lambda &= (x^{-1} \rhd \gamma \lhd e) \ast (y^{-1} \rhd \gamma^{\circ} \lhd e) \\
&= (x^{-1} \rhd \gamma \lhd e) \circ (x^{-1}y \rhd y^{-1} \rhd \gamma^{\circ} \lhd e) \\
&= (x^{-1} \rhd \gamma \lhd e) \circ (x^{-1} \rhd \gamma^{\circ} \lhd e) \\
&= x^{-1} \rhd (\gamma \circ  \gamma^{\circ}) \lhd e \\
&= x^{-1} \rhd 1_x \lhd e = 1_e \,.
\end{align*}
Hence if $\alpha$ and $\beta$ are paths differing by a cancelling pair of edges in $\Sq(\P)$ then $\alpha \lambda = \beta \lambda$ in the group
$(\st^{\x}_e,\ast)$.

Now consider a $2$--cell in the component of $\Sq(\P)$ containing $e$:
\begin{equation}
\label{general_2-cell_2}
\xymatrix@=6em{
\bullet \ar[d]_{(p(l \rho \tau)qt,  s,d, u)} \ar[r]^{(p, l,r, qt(s \rho \tau)u)} & \bullet \ar[d]^{(p(r \rho \tau)qt, s,d, u)} \\
\bullet \ar[r]_{(p, l,r, qt(d \rho \tau)u)} & \bullet}
\end{equation}
with
\begin{align*}
 \alpha &=(p, l,r, qt(s \rho \tau)u), \beta=(p(r \rho \tau)qt,s,d,u), \gamma=(p(l \rho \tau)qt,s,d,u), 
\\ \delta &= (p, l,r, qt(d \rho \tau)u) \,. \end{align*}
Then by Lemma \ref{inv_in_star}, $\alpha \lambda =  \lambda^e_{l,r,qt(s \rho \tau)u}$,
$\beta \lambda = \lambda^e_{s,d,u} = \gamma \lambda$ and $\delta \lambda = \lambda^e_{l,r,qt(d \rho \tau)u}$.
Hence path homotopy induced by the above $2$--cell in $\Sq(\P)$ is equivalent to the relation
\[ \lambda^e_{l,r,v(s \rho \tau)u} \ast \lambda^e_{s,d,u} = \lambda^e_{s,d,u} \ast \lambda^e_{l,r,v(d \rho \tau)u} \]
(where $v=qt$ above).  These considerations show that:

\begin{prop}
\label{def_rels_for_im_star}
Given $e \in E(\curlyT), q \in \curlyT$ and $(l,r) \in \curlyR$ with $e \leq  q^{-1}((l^{-1}l)\rho \tau)q$, we set
$\lambda^e_{l,r,q} = (eq^{-1}(l^{-1} \rho \tau),l,r,qe)$.  Then the following are a set of defining relations for the group
$(\st_e^{\x},\ast)$ on the generating set $\Sigma_e^{\x}$:
\[ \lambda^e_{l,r,v(s \rho \tau)u} \ast \lambda^e_{s,d,u} = \lambda^e_{s,d,u} \ast \lambda^e_{l,r,v(d \rho \tau)u} \, . \]
\end{prop}

\subsection{A crossed module from an inverse monoid presentation}
\label{cm_from_im}
As in Example \ref{trace_gpd}, we regard $\curlyT$ as a groupoid $\vec{\curlyT}$ (although we shall drop the arrow superscript hereon)
with vertex set $E=E(\curlyT)$, and we define 
\[ S^{\x} = \bigsqcup_{e \in E} \starsq \,.\]

\begin{prop}
\label{pres_cm_gpd_1}
$S^{\x} \xrightarrow{\R} \curlyT \rightrightarrows E$ is a crossed module of groupoids.
\end{prop}

\begin{proof}
By Lemma \ref{st_gp}, each $\st_e^{\x}$ is a group, and so $S^{\x}$ is a disjoint union of groups indexed by $E$, and is a groupoid with vertex set $E$. 
Then $\R$ is a groupoid homomorphism, and is the identity on $E$.

An action of $\curlyT$ on $S^{\x}$ is defined using the actions in the
pseudoregular groupoid $\pi(\Sq(\curlyP),\curlyT)$ as follows: 
for
$w \in \curlyT$ and $\alpha \in \st_{ww^{-1}}^{\x}$ we define
$$\alpha^w = w^{-1} \rhd \alpha \lhd w \in  \st_{w^{-1}w}^{\x} \,.$$
Then CM1 holds, since 
$(\alpha^w)\R = (w^{-1} \rhd \alpha \lhd w) \R= w^{-1} (\alpha \R) w$.
For CM2, since the binary operations $\ast$ and $\circledast$ on $\pi(Sq(\curlyP),\curlyT)$ coincide by Proposition \ref{monoidal_Pi_invmon}, then
$$\alpha \ast \beta = \alpha \circledast \beta = \beta \circ (\alpha \lhd \beta \R)$$
and
$$\beta \ast \alpha^{\beta \R} = \beta \ast ((\beta \R)^{-1} \rhd \alpha \lhd \beta \R) = \beta \circ (\alpha \lhd \beta \R) \,.$$
So $\alpha \ast \beta = \beta \ast \alpha^{\beta \R}$.
Therefore $S^{\x} \xrightarrow{\R} \curlyT \rightrightarrows E$ is a crossed module of groupoids.
\end{proof}

We shall now show that the crossed module $S^{\x} \xrightarrow{\R} \curlyT \rightrightarrows E$ is free, and give an explicit basis.
To do this, we give a construction of a free crossed $\curlyT(M,X)$--module
directly  from an inverse monoid presentation
$\curlyP = [X: \curlyR]$ of $M$ and show that it is isomorphic to the one in Proposition \ref{pres_cm_gpd_1}.   

Suppose that $(l,r) \in \curlyR$, and set $h =l \rho \tau$ and $k = r \rho \tau$.  Then
$h\psi = k\psi$ and so $(h^{-1}k) \psi \in E(M)$. Since $\psi : \curlyT \ra M$ is idempotent separating,  $(h^{-1}k) \psi = x \psi$ for a unique $x \in E(\curlyT)$ with
\[ (h^{-1}k) \psi = (h^{-1}h)\psi = (k^{-1}k) \psi \quad \text{and} \quad  (hk^{-1}) \psi = (hh^{-1})\psi = (kk^{-1}) \psi \,.\]
Hence $h^{-1}h= x = k^{-1}k$ and $hh^{-1}=kk^{-1}$.  So for $x \in E(\curlyT)$ we  define
\[
R_x =  \{ (l,r,x) \in \curlyR \times E : (l^{-1}r)\rho \tau \psi \geq x \psi \} \,, \]
and  consider the set $R = \bigsqcup_{x \in E} R_x$, 
along with the function $\omega: R \rightarrow \curlyT$ which maps $(l,r,x) \mapsto x((l^{-1}r) \rho \tau)x$. 
Then
\[ (l,r,x) \omega \D = xh^{-1}kxk^{-1}hx \; \text{and} \; (l,r,x) \omega \R = xk^{-1}hxh^{-1}kx \,, \]
with
\[ ((l,r,x) \omega \D) \psi  = (xh^{-1}kxk^{-1}hx) \psi = x \psi  =(xk^{-1}hxh^{-1}kx) \psi = ((l,r,x) \omega \R) \psi \,. \]
Since $\psi$ is idempotent separating, we conclude that $(l,r,x) \omega \D = (l,r,x) \omega \R$.
The free crossed $\curlyT$-module $C \xrightarrow{\partial} \curlyT \rightrightarrows E(M)$ with basis $\omega$
is then constructed as in Proposition \ref{free_pcm_gpds}. 

\begin{thm}
\label{pres_fcm_gpd}
The crossed $\curlyT$--module
$S^{\x} \xrightarrow{\R} \curlyT \rightrightarrows E(M)$
is isomorphic to  the free crossed $\curlyT$--module 
$C \xrightarrow{\partial} \curlyT \rightrightarrows E(M)$.
\end{thm}

\begin{proof}
For $(l,r,x) \in R$, we retain the notation $h = l \rho \tau$ and $k = r \rho \tau$.
We define $\nu: R \rightarrow S^{\x}$ by $(l,r,x) \mapsto (xh^{-1}, l,r, x)$.  We then have
$(xh^{-1}, l,r, x)\D = xh^{-1}hx = x$, and so   $(l,r,x)\nu \in \st_x^{\x}$.  Moreover, 
\[ (l,r,x)\nu \R = (xh^{-1},l,r,x) \R =x h^{-1}kx = (l,r,x) \omega \,.\]
Therefore $\nu\R=\omega$ and by freeness of $C \xrightarrow{\partial} \curlyT \rightrightarrows E(M)$ there exists a crossed module morphism  $\eta : C \rightarrow S^{\x}$ mapping 
\[ (l,r, u) \mapsto  (u^{-1}h^{-1}, l,r, u) \in \st^{\x}(u^{-1}u) \,, \]
where $(h^{-1}k) \psi \geq  (uu^{-1})\psi$ and $h^{-1}h = k^{-1}k \geq uu^{-1}$.
We claim that $\eta$ is an isomorphism, and we verify this by constructing its inverse.

We define
$\mu: \Sigma_e^{\x} \ra C_e$ by $\mu : \lambda^e_{l,r,q} \mapsto (l,r, qe)$, where 
$\lambda^e_{l,r,q}$ is defined in Proposition \ref{def_rels_for_im_star}.  We
consider the effect of this map on a defining relation 
\[ \lambda^e_{l,r,v(s \rho \tau)u} \ast \lambda^e_{s,d,u} = \lambda^e_{s,d,u} \ast \lambda^e_{l,r,v(d \rho \tau)u} \, . \]
as given in Proposition \ref{def_rels_for_im_star}.  We set $g = s \rho \tau$ and $t = d \rho \tau$.  Then
\[  
\lambda^e_{l,r,vgu} \stackrel{\mu}{\mapsto} (l,r,vgue), \;
\lambda^e_{s,d,u}  \stackrel{\mu}{\mapsto} (s,d,ue), \; \text{and} \;
\lambda^e_{l,r,vtu}  \stackrel{\mu}{\mapsto} (l,r,vtue) \,.
\]
In the group $C_e$ we have
\[ (s,d,ue)^{-1}(l,r,vgue)(s,d,ue) = (l,r,vgueu^{-1}g^{-1}tue) \,. \]
Now $g \psi = t \psi$ and so
\[ [(gue)(gue)^{-1}] \psi = [(tue)(tue)^{-1}]\psi \,.\]
Since $\psi$ is idempotent separating, $(gue)(gue)^{-1} = (tue)(tue)^{-1}$ and therefore 
\[ vgueu^{-1}g^{-1}tue = v(gue)(gue)^{-1}tue = v(tue)(tue)^{-1}(tue) = vtue \,.\]
So in $C_e$ we have 
\[ (s,d,ue)^{-1}(l,r,vgue)(s,d,ue) = (l,r,vtue) \]
and $\mu$ induces a homomorphism $\st^{\x}_e \ra C_e$ that is the inverse of $\eta$.
\end{proof}

\noindent
As a module for the groupoid $\vec{M}$, we see by Proposition \ref{free_cm_and_free_mod} that
$(S^{\x})^{ab}$ is free, with basis function
$(l,r,u) \mapsto \ol{(u^{-1}l^{-1},l,r,u)}$.
However, we can say more.

\begin{prop}
\label{free_log_mod}
$(S^{\x})^{ab}$ is the free $M$--module on the $E(M)$--set $\curlyZ$ in which
$Z_e = \{(l,r) \in R : (l^{-1}r) \rho \tau\psi = e \psi \}$.
\end{prop}

\begin{proof}
The groupoid action of $\vec{M}$ extends to one of $M$.  If $\alpha \in \st^{\x}_e$ with image $\ol{\alpha} \in (\st^{\x}_e)^{ab}$, 
and $m \in M$ with $m = w \psi$, then we define 
\[ \ol{\alpha} \lhd m = \ol{w^{-1} \rhd \alpha \lhd w} \,. \]
As a component of the free $\vec{M}$--module $(S^{\x})^{ab}$, the group $(\st^{\x}_e)^{ab}$ is the free abelian
group with basis
\[ Y_e = \{ (l,r,m) : (l^{-1}r) \rho \tau \psi \geq mm^{-1}, m^{-1}m=e \} \,.\]
which is the correct basis for $(S^{\x})^{ab}$ as the free $M$--module on the $E(M)$--set $\curlyZ$.
\end{proof}

\subsection{A presentation for the relation module}
From an inverse monoid presentation $\P = [X:\curlyR]$ of an inverse monoid $M$ we have now constructed a free crossed module
$S^{\bowtie}  \xrightarrow{\R} \curlyT \rightrightarrows E$ and for each $e \in E$ we have a
crossed module of groups $\st^{\x}_e \xrightarrow{\R} K_e$.  Since $K_e$ is the vertex set of the component of
$\Sq(\P)$ containing $e$, the map $\R: \st^{\x}_e \ra K_e$ is surjective.  
By Propositions \ref{piex} and \ref{monoidal_Pi_invmon},
\[ \pi_e^{\x}= \{ \alpha \in \st^{\x}_e : \alpha \R = e \} \]
is  abelian and we have a short exact sequence of groups
\begin{equation}
\label{Ue_ses}
 0 \ra \pi_e^{\x} \ra \st^{\x}_e \ra K_e \ra 1 \,.
\end{equation}

\begin{lemma}
\label{Ue_splitting}
Each group $K_e$ is free, the sequence \eqref{Ue_ses} splits and, $\st^{\x}_e$ and $\pi_e^{\x} \times K_e$ are isomorphic
groups.
\end{lemma}

\begin{proof}
The group $K_e$ is a subgroup of $\curlyT$ and the maximum group image map $\sigma : \curlyT \ra F(X)$ is idempotent
pure.  Its restriction $\sigma : K_e \ra F(X)$  therefore has trivial kernel and so $K_e$ is isomorphic to a subgroup of a 
free group and is free.  By \eqref{cm_gpd_2}, $K_e$ acts trivially on $\pi_e^{\x}$ and so the splitting of the sequence 
\eqref{Ue_ses} induces an isomorphism $\st^{\x}_e \cong \pi_e^{\x} \times K_e$.
\end{proof}

\begin{thm}
\label{relmod_pres_im}
Let $\P$ be an inverse monoid presentation of an inverse monoid $M$.  There exists a short exact sequence of $M$--modules
\begin{equation}
\label{relmod_ses}
0 \ra \bigsqcup_{e \in E(M)} \pi_e^{\bowtie} \ra (S^{\bowtie})^{ab} \xrightarrow{\ol{\R}}  \curlyK \ra 0 
\end{equation}
in which $(S^{\bowtie})^{ab}$ is a free $M$--module and $\ol{\R}$ is induced by $S^{\bowtie}  \xrightarrow{\R} K$.
\end{thm}

\begin{proof}
The $M$--module structure on $\bigsqcup_{e \in E(M)} \pi_e^{\bowtie}$ is given by Proposition \ref{piex}, that on 
$(S^{\x})^{ab}$ by Proposition \ref{free_log_mod}, and that on $\curlyK$ by Lemma \ref{ker_is_mod}.
Lemma \ref{Ue_splitting} gives us, for each $e \in E(M)$, a short exact sequence of abelian groups
\[ 0 \ra \pi_e^{\x} \ra \st^{\x}_e \xrightarrow{\ol{\partial}} K_e^{ab} \ra 0 \]
and these assemble into the sequence \eqref{relmod_ses}.  It remains to check that $\ol{\R}$ is then a map of
$M$--modules.  

Let $\alpha \in \st^{\x}_e$ with image $\ol{\alpha} \in (S^{\x})^{ab}$, and let $m \in M$ with $mm^{-1}=e$.
Then the action of $m$ on $\ol{\alpha}$ is defined by lifting $m$ to $\curlyT$ and acting on $\alpha$ in the crossed module
$S^{\bowtie}  \xrightarrow{\R} \curlyT$:
\[ \ol{\alpha} \lhd m = \ol{t^{-1} \rhd \alpha \lhd t} \]
where $t \psi =m$.  Hence
\[ (\ol{\alpha} \lhd m)  \ol{\R} = \ol{( t^{-1} \rhd \alpha \lhd t) \R} = \ol{t^{-1}(\alpha \R)t} \in U_{m^{-1}m}^{ab} \,.\]
But $\ol{\alpha} \; \ol{\R} = \ol{\alpha \R}$ and $\ol{\alpha \R} \lhd m = \ol{t^{-1}(\alpha \R)t}$.
 \end{proof}

\subsection*{Acknowledgements}
A version of these results is presented in the second author's PhD thesis at Heriot-Watt University, Edinburgh.  The generous
financial support of a PhD Scholarship from the Carnegie Trust for the Universities of Scotland is duly and gratefully acknowledged.

\end{document}